\documentclass[12pt, a4paper]{article}
\usepackage{amsfonts}
\usepackage{pdfsync}
\usepackage{amsmath}
\usepackage[latin1]{inputenc}
\usepackage{amsthm}
\usepackage{amssymb}
\usepackage[francais,english]{babel}

\numberwithin{equation}{section}
\theoremstyle{plain}
\newtheorem{thm}{Theorem}[section]
\newtheorem{coro}[thm]{Corollary}
\newtheorem{prop}[thm]{Proposition}
\newtheorem{lem}[thm]{Lemma}
\newtheorem{defi}[thm]{Definition}
\theoremstyle{definition}

\theoremstyle{remark}

\newcommand{\Perp}{\perp \! \! \! \perp}

\providecommand\mathbb{\bf}

\newcommand\R{{\mathbb R}}
\newcommand\E{{\mathbb E}}

\newcommand\pref[1]{(\ref{#1})}

\newcommand\indi{{\bf 1}}
\newcommand\qtext[1]{\quad\text{#1}\quad}

\let \eps\varepsilon
\let \phi\varphi
\newcommand{\PP}{\mathbb{P}}

\newcommand\CC{{\cal C}}

\newcommand\tX{\widetilde{X}}
\newcommand\tY{\widetilde{Y}}
\newcommand\tU{\widetilde{U}}
\newcommand\oU{\overline{U}}
\newcommand\oal{\overline{\alpha}}
\newcommand\obe{\overline{\beta}}

\newcommand\Law{\mathop{\mathrm{Law}}\nolimits}

\def\<#1,#2>{\left<#1,#2\right>}

\begin{document}

\title{Vector quantile regression beyond correct specification}
\author{G. Carlier\thanks{{\scriptsize CEREMADE, UMR CNRS 7534, Universit\'e
Paris IX Dauphine, Pl. de Lattre de Tassigny, 75775 Paris Cedex 16, FRANCE,
and MOKAPLAN Inria Paris, \texttt{carlier@ceremade.dauphine.fr}}}, V.
Chernozhukov \thanks{{\scriptsize Department of Economics, MIT, 50 Memorial
Drive, E52-361B, Cambridge, MA 02142, USA, \texttt{vchern@mit.edu}}}, A.
Galichon \thanks{{\scriptsize Economics Department and Courant Institute of
Mathematical Sciences, NYU, 70 Washington Square South, New York, NY 10013,
USA \texttt{ag133@nyu.edu}.}}}
\maketitle

\begin{abstract}
This paper studies vector quantile regression (VQR), which is a way to model
the dependence of a random vector of interest with respect to a vector of
explanatory variables so to capture the whole conditional distribution, and
not only the conditional mean. The problem of vector quantile regression is
formulated as an optimal transport problem subject to an additional
mean-independence condition. This paper provides a new set of results on VQR
beyond the case with correct specification which had been the focus of previous work. First,
we show that even under misspecification, the VQR\ problem still has a
solution which provides a general representation of the conditional
dependence between random vectors. Second, we provide a detailed comparison
with the classical approach of Koenker and Bassett in the case when the
dependent variable is univariate and we show that in that case, VQR is
equivalent to classical quantile regression with an additional monotonicity
constraint.
\end{abstract}


\section{Introduction}

\label{intro}

Vector quantile regression was recently introduced in \cite{ccg} in order to
generalize the technique of quantile regression when the dependent random
variable is multivariate. Quantile regression, pioneered by Koenker and
Bassett \cite{kb}, provides a powerful way to study dependence between
random variables assuming a linear form for the quantile of the endogenous
variable $Y$ given the explanatory variables $X$. It has therefore become a
very popular tool in many areas of economics, program evaluation,
biometrics, etc. However, a well-known limitation of the approach is that $Y$
should be scalar so that its quantile map is defined. When $Y$ is
multivariate, there is no canonical notion of quantile, and the picture is
less clear than in the univariate case\footnote{%
There is actually an important literature that aims at generalizing the
notion of quantile to a multidimensional setting and various different
approaches have been proposed; see in particular \cite{belloni}, \cite%
{hallin}, \cite{PuccettiScarsini} and the references therein.}. The approach
proposed in \cite{ccg} is based on optimal transport ideas and can be
described as follows. For a random vector vector $Y$ taking values in ${%
\mathbb{R}}^{d}$, we look for a random vector $U$ uniformly distributed on
the unit cube $[0,1]^{d}$ and which is maximally correlated to $Y$, finding
such a $U$ is an optimal transport problem. A celebrated result of Brenier
\cite{brenier} implies that such an optimal $U$ is characterized by the
existence of a convex function $\varphi $ such that $Y=\nabla \varphi (U)$.
When $d=1$, of course, the optimal transport map of Brenier $\nabla \varphi
=Q$ is the quantile of $Y$ and in higher dimensions, it still has one of the
main properties of univariate quantiles, namely monotonicity. Thus Brenier's
map $\nabla \varphi $ is a natural candidate to be considered as the vector
quantile of $Y$, and one advantage of such an approach is the pointwise
relation $Y=\nabla \varphi (U)$ where $U$ is a uniformly distributed random
vector which best approximates $Y$ in $L^{2}$.

\smallskip

If, in addition, we are given another random vector $X$ capturing a set of
observable explanatory variables, we wish to have a tractable method to
estimate the conditional quantile of $Y$ given $X=x$, that is the map $u\in
\lbrack 0,1]^{d}\mapsto Q(x,u)\in {\mathbb{R}}^{d}$. In the univariate case $%
d=1$, and if the conditional quantile is affine in $x$ i.e. $Q(x,u)=\alpha
(u)+\beta (u)x$, the quantile regression method of Koenker and Bassett gives
a constructive and powerful linear programming approach to compute the
coefficients $\alpha (t)$ and $\beta (t)$ for any fixed $t\in \lbrack 0,1]$,
dual to the linear programming problem:
\begin{equation}
\sup_{(U_{t})}\{{\mathbb{E}}(U_{t}Y)\;:\;U_{t}\in \lbrack 0,1],{\mathbb{E}}%
(U_{t})=(1-t),\;{\mathbb{E}}(XU_{t})={\mathbb{E}}(X)\}.  \label{kobas}
\end{equation}%
Under correct specification, i.e. when the true conditional
quantile is affine in $x$, this variational approach estimates the true
coefficients $\alpha (t)$ and $\beta (t)$. In \cite{ccg}, we have shown that
in the multivariate case as well, when the true vector quantile is affine in
$x$, one may estimate it by a variational problem which consists in finding
the uniformly distributed random variable $U$ such that ${\mathbb{E}}(X|U)={%
\mathbb{E}}(X)$ (mean independence) and maximally correlated with $Y$.

\smallskip

The purpose of the present paper is to understand what these variational
approaches tell about the dependence between $Y$ and $X$ in the general case
i.e. without assuming any particular form for the conditional quantile. Our
main results are the following:

\begin{itemize}
\item \textbf{A general representation of dependence:} we will characterize
the solution of the optimal transport problem with a mean-independence
constraint from \cite{ccg} and relate it to a relaxed form of vector
quantile regression. To be more precise, our theorem \ref{qrasfoc} below
will provide the following general representation of the distribution of $%
\left( X,Y\right) $:
\begin{equation*}
\begin{split}
Y& \in \partial \Phi _{X}^{\ast \ast }(U)%
\mbox{ with $X\mapsto \Phi_X(U)$
affine}, \\
& \Phi _{X}(U)=\Phi _{X}^{\ast \ast }(U)\mbox{ almost surely,} \\
& \mbox{ $U$ uniformly distributed on $[0,1]^d$},\;{\mathbb{E}}(X|U)={%
\mathbb{E}}(X),
\end{split}%
\end{equation*}%
where $\Phi _{x}^{\ast \ast }$ denotes the convex envelope of $u\rightarrow
\Phi _{x}\left( u\right) $ for a fixed $x$, and $\partial $ denotes the
subdifferential. The main ingredients are convex duality and an existence
theorem for optimal dual variables. The latter is a non-trivial extension of
Kantorovich duality: indeed, the existence of a Lagrange multiplier
associated to the mean-independence constraint is not straightforward and we
shall prove it thanks to Komlos' theorem (theorem \ref{existdual}). Vector
quantile regression is \emph{under correct specification} if $\Phi _{x}\left( u\right) $ is
convex for all $x$ in the support, in which case one can write
\begin{equation*}
\begin{split}
Y& =\nabla \Phi _{X}(U)%
\mbox{ with $\Phi_X(.)$ convex, $X\mapsto \Phi_X(U)$
affine}, \\
& \mbox{ $U$ uniformly distributed on $[0,1]^d$},\;{\mathbb{E}}(X|U)={%
\mathbb{E}}(X).
\end{split}%
\end{equation*}%
While our previous paper~\cite{ccg} focused on the case with correct specification, the
results we obtain in the present paper are general.

\item \textbf{A precise link with classical quantile regression in the
univariate case:} it was shown in~\cite{ccg} that in the particular case
when $d=1$ and under correct specification, classical quantile
regression and vector quantile regression are equivalent. Going beyond correct specification here, we shall see that the optimal transport approach is
equivalent (theorem \ref{equivkb}) to a variant of (\ref{kobas}) where one
further imposes the monotonicity constraint that $t\mapsto U_{t}$ is
nonincreasing (which is consistent with the fact that the true quantile $%
Q(x,t)$ is nondecreasing with respect to $t$).
\end{itemize}

\smallskip

The paper is organized as follows. In section~\ref{mvquant}, introduces
vector quantiles through optimal transport. Section \ref{mvqr} is devoted to
a precise, duality based, analysis of the vector quantile regression beyond
correct specification. Finally, we shall revisit in section \ref{univar} the
univariate case and then carefully relate the Koenker and Bassett approach
to that of \cite{ccg}.

\section{Vector quantiles and optimal transport}

\label{mvquant}

Let $(\Omega ,{\mathcal{F}},\mathbb{P})$ be some nonatomic probability
space, and let $\left( X,Y\right) $ be a random vector, where the vector of
explanatory variables $X$ is valued in ${\mathbb{R}}^{N}$ and the vector of
dependent variables $Y$ is valued in ${\mathbb{R}}^{d}$.

\subsection{Vector quantiles by correlation maximization}

\label{mvquant1}

The notion of vector quantile was recently introduced by Ekeland, Galichon
and Henry \cite{egh}, Galichon and Henry \cite{gh} and was used in the
framework of quantile regression in our companion paper \cite{ccg}. The
starting point for this approach is the correlation maximization problem
\begin{equation}
\max \{{\mathbb{E}}(V\cdot Y),\;\mathop{\mathrm{Law}}\nolimits(V)=\mu \}
\label{otmd}
\end{equation}%
where $\mu :=\mathop{\mathrm{uniform}}\nolimits([0,1]^{d})$ is the uniform
measure on the unit $d$-dimensional cube $[0,1]^{d}$. This problem is
equivalent to the optimal transport problem which consists in minimizing ${%
\mathbb{E}}(|Y-V|^{2})$ among uniformly distributed random vectors $V$. As
shown in the seminal paper of Brenier \cite{brenier}, this problem has a
solution $U$ which is characterized by the condition
\begin{equation*}
Y=\nabla \varphi (U)
\end{equation*}%
for some (essentially uniquely defined) convex function $\varphi $ which is
again obtained by solving a dual formulation of (\ref{otmd}). Arguing that
gradients of convex functions are the natural multivariate extension of
monotone nondecreasing functions, the authors of \cite{egh} and \cite{gh}
considered the function $Q:=\nabla \varphi $ as the vector quantile of $Y$.
This function $Q=\nabla \varphi $ is by definition the Brenier's map, i.e.
the optimal transport map (for the quadratic cost) between the uniform
measure on $[0,1]^{d}$ and $\mathop{\mathrm{Law}}\nolimits(Y)$:

\begin{thm}
\textbf{(Brenier's theorem)} If $Y$ is a squared-integrable random vector
valued in $\mathbb{R}^{d}$, there is a unique map of the form $T=\nabla
\varphi $ with $\varphi $ convex on $[0,1]^{d}$ such that $\nabla \varphi
_{\#}\mu =\mathop{\mathrm{Law}}\nolimits(Y)$, this map is by definition the
vector quantile function of $Y$.
\end{thm}

We refer to the textbooks \cite{villani}, \cite{villani2} and \cite{sanbook}
for a presentation of optimal transport theory, and to~\cite{otme} for a
survey of applications to economics.

\subsection{Conditional vector quantiles}

\label{mvquant2}

Take a $N$-dimensional random vector $X$ of regressors, $\nu :=%
\mathop{\mathrm{Law}}\nolimits(X,Y)$, $m:=\mathop{\mathrm{Law}}\nolimits(X)$%
, $\nu =\nu ^{x}\otimes m$ where $m$ is the law of $X$ and $\nu ^{x}$ is the
law of $Y$ given $X=x$. One can consider $Q(x,u)=\nabla \varphi (x,u)$ as
the optimal transport between $\mu $ and $\nu ^{x}$, under some regularity
assumptions on $\nu ^{x}$, one can invert $Q(x,.)$: $Q(x,.)^{-1}=\nabla
_{y}\varphi (x,.)^{\ast }$ (where the Legendre transform is taken for fixed $%
x$) and one can define $U$ through
\begin{equation*}
U=\nabla _{y}\varphi ^{\ast }(X,Y),\;Y=Q(X,U)=\nabla _{u}\varphi (X,U)
\end{equation*}%
$Q(X,.)$ is then the conditional vector quantile of $Y$ given $X$. There is,
as we will see in dimension one, a variational principle behind this definition:

\begin{itemize}
\item $U$ is uniformly distributed, independent from $X$ and solves:

\begin{equation}  \label{otmdindep}
\max \{ {\mathbb{E}}(V\cdot Y), \; \mathop{\mathrm{Law}}\nolimits(V)=\mu,
V\perp \! \! \! \perp X\}
\end{equation}

\item the conditional quantile $Q(x,.)$ and its inverse are given by $%
Q(x,u)=\nabla_u \varphi(x,u)$, $F(x,y)=\nabla_y \psi(x,y)$ (the link between
$F$ and $Q$ being $F(x, Q(x,u))=u$), the potentials $\psi$ and $\varphi$ are
convex conjugates ($x$ being fixed) and solve

\begin{equation*}
\min \int \varphi(x,u) m(dx) \mu(du) + \int \psi(x,y) \nu(dx, dy) \; : \;
\psi(x,y)+\varphi(x,u)\ge y\cdot u.
\end{equation*}
\end{itemize}


Note that if the conditional quantile function is affine in $X$ and $%
Y=Q(X,U)=\alpha(U) +\beta(U) X$ where $U$ is uniform and independent from $X$%
, the function $u\mapsto \alpha(u)+\beta(u)x$ should be the gradient of some
function of $u$ which requires
\begin{equation*}
\alpha=\nabla \varphi, \; \beta=Db^T
\end{equation*}
for some potential $\varphi$ and some vector-valued function $b$ in which
case, $Q(x,.)$ is the gradient of $u\mapsto \varphi(u)+b(u)\cdot x$.
Moreover, since quantiles are gradients of convex potentials one should also
have
\begin{equation*}
u\in [0,1]^d \mapsto \varphi(u) +b(u) \cdot x \mbox{ is convex}.
\end{equation*}

\section{Vector quantile regression}

\label{mvqr}


In the next paragraphs, we will impose a parametric form of the dependence
of the dependence of the vector quantile $Q(x,u)$ with respect to $x$. More
specifically, we shall assume that $Q(x,u)$ is affine in $x$. In the scalar
case ($d=1$), this problem is called quantile regression; we shall
investigate that case in section~\ref{univar} below.

\subsection{Correlation maximization}

\label{mvqr1}

Without loss of generality we normalize $X$ so that it is centered
\begin{equation*}
{\mathbb{E}}(X)=0.
\end{equation*}%
Our approach to vector quantile regression is based on a variation of the
correlation maximization problem~\ref{otmdindep}, where the independence
constraint is replaced by a mean-independence constraint, that is%
\begin{equation}
\max \{{\mathbb{E}}(V\cdot Y),\;\mathop{\mathrm{Law}}\nolimits(V)=\mu ,\;{%
\mathbb{E}}(X|V)=0\}.  \label{maxcorrmimd}
\end{equation}%
where $\mu =\mathop{\mathrm{uniform}}\nolimits([0,1]^{d})$ is the uniform
measure on the unit $d$-dimensional cube.

An obvious connection with the specification of vector quantile regression
(i.e. the validity of an affine in $x$ form for the conditional quantile) is
given by:

\begin{prop}
If $Y=\nabla \varphi(U)+ Db(U)^T X$ with

\begin{itemize}
\item $u\mapsto \varphi(u)+ b(u)\cdot x$ convex and smooth for $m$-a.e $x$,

\item $\mathop{\mathrm{Law}}\nolimits(U)=\mu$, ${\mathbb{E}}(X\vert U)=0$,
\end{itemize}

then $U$ solves (\ref{maxcorrmimd}).
\end{prop}

\begin{proof} This result follows from \cite{ccg}, but for the sake of completeness, we give a proof:
\[Y=\nabla \Phi_X(U), \mbox{ with } \Phi_X(t)=\varphi(t)+b(X)\cdot t.\]
Let $V$ be such that $\Law(V)=\mu$,  $\E(X\vert V)=0$, then by Young's inequality
\[V\cdot Y \le  \Phi_X(V)+\Phi_X^*(Y)\]
but $Y=\nabla \Phi_X(U)$ implies that
\[U\cdot Y =  \Phi_X(U)+\Phi_X^*(Y)\]
so taking  expectations gives the desired result.

\end{proof}

\subsection{Duality}

\label{mvqr2}

From now on, we do not assume a particular form for the conditional quantile
and wish to study which information (\ref{maxcorrmimd}) can give regarding
the dependence of $X$ and $Y$. Once again, a good starting point is convex
duality. As explained in details in \cite{ccg}, the dual of (\ref%
{maxcorrmimd}) takes the form
\begin{equation}  \label{dualmimc}
\inf_{(\psi, \varphi, b)} {\mathbb{E}}(\psi(X,Y)+\varphi(U)) \; : \;
\psi(x,y)+\varphi(t)+b(t)\cdot x\ge t\cdot y.
\end{equation}
where $U$ is any uniformly distributed random vector on $[0,1]^d$ i.e. $%
\mathop{\mathrm{Law}}\nolimits(U)=\mu=\mathop{\mathrm{uniform}}%
\nolimits([0,1]^d)$ and the infimum is taken over continuous functions $%
\psi\in C(\mathop{\mathrm{spt}}\nolimits(\nu), {\mathbb{R}})$, $\varphi \in
C([0,1]^d, {\mathbb{R}})$ and $b\in C([0,1]^d, {\mathbb{R}}^N)$ satisfying
the pointwise constraint
\begin{equation}  \label{constraintdudual}
\psi(x,y)+\varphi(t)+b(t)\cdot x\ge t\cdot y, \; \; \forall (x,y,t)\in %
\mathop{\mathrm{spt}}\nolimits(\nu)\times [0,1]^d.
\end{equation}

Since for fixed $(\varphi, b)$, the largest $\psi$ which satisfies the
pointwise constraint in (\ref{dualmimc}) is given by the convex function
\begin{equation*}
\psi(x,y):=\max_{t\in [0,1]^d} \{ t\cdot y- \varphi(t) -b(t)\cdot x\}
\end{equation*}
one may equivalently rewrite (\ref{dualmimc}) as the minimimization over
continuous functions $\varphi$ and $b$ of
\begin{equation*}
\int \max_{t\in [0,1]} \{ ty- \varphi(t) -b(t)\cdot x\} \nu(dx,dy)+\int_{[0,
1]^d} \varphi(t)\mu(dt).
\end{equation*}
We claim now that the infimum over continuous functions $(\varphi, b)$
coincides with the one over smooth or simply integrable functions. Indeed,
let $b\in L^1((0,1)^d)^N$, $\varphi \in L^1((0,1)^d))$ and $\psi$ such that (%
\ref{constraintdudual}) holds. Let $\eps>0$ and, extend $\varphi$ and $b$ to
$Q_\eps:=[0,1]^d+B_{\eps}$ ($B_\eps$ being the closed Euclidean ball of
center $0$ and radius $\eps$):
\begin{equation*}
\varphi_\eps(t):=%
\begin{cases}
\varphi(t), \mbox{ if $t\in (0,1)^d$} \\
\frac{1}{\eps} \mbox{ if $t\in Q_\eps \setminus (0,1)^d$}%
\end{cases}%
, \; b_\eps(t):=%
\begin{cases}
b(t), \mbox{ if $t\in (0,1)^d$} \\
0 \mbox{ if $t\in Q_\eps \setminus (0,1)^d$}%
\end{cases}%
\end{equation*}
and for $(x,y)\in \mathop{\mathrm{spt}}\nolimits(\nu)$:
\begin{equation*}
\psi_\eps(x,y):=\max\Big(\psi(x,y), \max_{t\in Q_\eps \setminus (0,1)^d}
(t\cdot y-\frac{1}{\eps})\Big)
\end{equation*}
then by construction $(\psi_\eps, \varphi_\eps, b_\eps)$ satisfies (\ref%
{constraintdudual}) on $\mathop{\mathrm{spt}}\nolimits(\nu)\times Q_\eps$.
Let $\rho \in C_c^{\infty}({\mathbb{R}}^d)$ be a centered, smooth
probability density supported on $B_1$, and define the mollifiers $%
\rho_\delta:=\delta^{-d} \rho(\frac{.}{\delta})$, then for $\delta \in (0, %
\eps)$, defining the smooth functions $b_{\eps, \delta}:=\rho_\delta \star b_%
\eps$ and $\varphi_{\eps, \delta}:=\rho_\delta \star \varphi_\eps$, we have
that $(\psi_\eps, \varphi_{\eps, \delta}, b_{\eps, \delta})$ satisfies (\ref%
{constraintdudual}). By monotone convergence, $\int \psi_\eps \mbox{d} \nu$
converges to $\int \psi \mbox{d}\nu$, moreover
\begin{equation*}
\lim_{\delta \to 0} \int_{(0,1)^d} \varphi_{\eps, \delta} = \int_{(0,1)^d}
\varphi_\eps=\int_{(0,1)^d} \varphi,
\end{equation*}
we deduce that the value of the minimization problem (\ref{dualmimc}), can
indifferently be obtained by minimizing over continuous, smooth or $L^1$ $%
\varphi$ and $b$'s. The existence of optimal ($L^1$) functions $\psi,
\varphi $ and $b$ is not totally obvious and is proven in the appendix under
the following assumptions:

\begin{itemize}
\item the support of $\nu$, is of the form $\mathop{\mathrm{spt}}%
\nolimits(\nu):=\overline{\Omega}$ where $\Omega$ is an open bounded convex
subset of ${\mathbb{R}}^N\times {\mathbb{R}}^d$,

\item $\nu\in L^{\infty}(\Omega)$,

\item $\nu$ is bounded away from zero on compact subsets of $\Omega$ that is
for every $K$ compact, included in $\Omega$ there exists $\alpha_K>0$ such
that $\nu\ge \alpha_K$ a.e. on $K$.
\end{itemize}

\begin{thm}
\label{existdual} Under the assumptions above, the dual problem (\ref%
{dualmimc}) admits at least a solution.
\end{thm}

\subsection{Vector quantile regression as optimality conditions}

\label{mvqr3}

Let $U$ solve (\ref{maxcorrmimd}) and $(\psi, \varphi, b)$ solve its dual (%
\ref{dualmimc}). Recall that, without loss of generality, we can take $\psi$
convex given
\begin{equation}  \label{psiconj}
\psi(x,y)=\sup_{t\in [0,1]^d} \{ t\cdot y- \varphi(t) -b(t)\cdot x\} .
\end{equation}
The constraint of the dual is
\begin{equation}  \label{contraintedual}
\psi(x,y)+\varphi(t)+b(t)\cdot x\ge t\cdot y, \; \forall (x,y,t)\in
\Omega\times [0,1]^d,
\end{equation}
and the primal-dual relations give that, almost-surely
\begin{equation}  \label{dualrel000}
\psi(X,Y)+\varphi(U)+b(U)\cdot X= U\cdot Y.
\end{equation}
Which, since $\psi$, given by (\ref{psiconj}) , is convex, yields
\begin{equation*}
(-b(U), U)\in \partial \psi(X,Y), \mbox{ or, equivalently } (X,Y)\in
\partial \psi^*(-b(U),U).
\end{equation*}

Problems (\ref{maxcorrmimd}) and (\ref{dualmimc}) have thus enabled us to
find:

\begin{itemize}
\item $U$ uniformly distributed with $X$ mean-independent from $U$,

\item $\phi$ : $[0,1]^d \to {\mathbb{R}}$, $b$ : $[0,1]^d \to {\mathbb{R}}^N$
and $\psi$ : $\Omega\to {\mathbb{R}}$ convex,
\end{itemize}

such that $(X,Y)\in \partial \psi ^{\ast }(-b(U),U)$. Specification of
vector quantile regression rather asks whether one can write $Y=\nabla
\varphi (U)+Db(U)^{T}X:=\nabla \Phi _{X}(U)$ with $u\mapsto \Phi
_{x}(u):=\varphi (u)+b(u)x$ convex in $u$ for fixed $x$. The smoothness of $%
\varphi $ and $b$ is actually related to this specification issue. Indeed,
if $\varphi $ and $b$ were smooth then (by the envelope theorem) we would
have
\begin{equation*}
Y=\nabla \varphi (U)+Db(U)^{T}X=\nabla \Phi _{X}(U).
\end{equation*}%
But even smoothness of $\varphi $ and $b$ are not enough to guarantee that
the conditional quantile is affine in $x$, which would also require $%
u\mapsto \Phi _{x}(u)$ to be convex. Note also that if $\psi $ was smooth,
we would then have
\begin{equation*}
U=\nabla _{y}\psi (X,Y),\;-b(U)=\nabla _{x}\psi (X,Y)
\end{equation*}%
so that $b$ and $\psi $ should be related by the vectorial Hamilton-Jacobi
equation
\begin{equation}
\nabla _{x}\psi (x,y)+b(\nabla _{y}\psi (x,y))=0.  \label{hj}
\end{equation}

In general (without assuming any smoothness), define
\begin{equation*}
\psi_x(y)=\psi(x,y).
\end{equation*}
We then have, thanks to (\ref{contraintedual})-(\ref{dualrel000})
\begin{equation*}
U\in \partial \psi_X(Y) \mbox{ i.e. } Y \in \partial \psi_X^* (U).
\end{equation*}
The constraint of (\ref{dualmimc}) also gives
\begin{equation*}
\psi_x(y) +\Phi_x(t)\ge t\cdot y
\end{equation*}
since Legendre Transform is order-reversing, this implies
\begin{equation}  \label{inegdualps}
\psi_x \ge \Phi_x^*
\end{equation}
hence
\begin{equation*}
\psi_x^* \le (\Phi_x)^{**} \le \Phi_x
\end{equation*}
(where $\Phi_x^{**}$ denotes the convex envelope of $\Phi_x$). Duality
between (\ref{maxcorrmimd}) and (\ref{dualmimc}) thus gives:

\begin{thm}
\label{qrasfoc} Let $U$ solve (\ref{maxcorrmimd}), $(\psi, \varphi, b)$
solve its dual (\ref{dualmimc}) and set $\Phi_x(t):=\varphi(t)+b(t)\cdot x$
for every $(t,x)\in [0,1]^d \times \mathop{\mathrm{spt}}\nolimits(m)$ then
\begin{equation}  \label{relaxspec}
\Phi_X(U)= \Phi_X^{**}(U) \mbox{ and } U \in \partial \Phi_X^*(Y)
\mbox{
i.e. } Y \in \partial \Phi_X^{**}(U).
\end{equation}
almost surely.
\end{thm}

\begin{proof}
From the duality relation \pref{dualrel000} and \pref{inegdualps}, we have
\[U\cdot Y=\psi_X(Y)+\Phi_X(U)\ge \Phi_X^*(Y)+\Phi_X(U)\]
so that $U\cdot Y= \Phi_X^*(Y)+\Phi_X(U)$ and then
\[\Phi_X^{**}(U)\ge U\cdot Y-\Phi_X^*(Y)=\Phi_X(U).\]
Hence, $\Phi_X(U)=\Phi_X^{**}(U)$  and $U\cdot Y=\Phi_X^*(Y)+\Phi_X^{**}(U)$ i.e. $U\in \partial \Phi_X^*(Y)$ almost surely, and the latter is equivalent to the requirement that  $Y \in \partial \Phi_X^{**}(U)$.

\end{proof}


The previous theorem thus gives the following interpretation of the
correlation maximization with a mean independence constraint (\ref%
{maxcorrmimd}) and its dual (\ref{dualmimc}). These two variational problems
in duality lead to the pointwise relations (\ref{relaxspec}) which can be
seen as best approximations of a specification assumption:
\begin{equation*}
Y=\nabla \Phi _{X}(U),\;(X,U)\mapsto \Phi _{X}(U)%
\mbox{ affine in $X$,
convex in $U$}
\end{equation*}%
with $U$ uniformly distributed and $\E(X\vert U)=0$. Indeed in (\ref{relaxspec}), $\Phi _{X}$ is replaced by its convex envelope,
the uniform random variable $U$ solving (\ref{maxcorrmimd}) is shown to lie
a.s. in the contact set $\Phi _{X}=\Phi _{X}^{\ast \ast }$ and the gradient of $\Phi_X$ (which may not be well-defined)  is replaced by a subgradient of $\Phi_X^{**}$.

\section{The univariate case}

\label{univar}

We now study in detail the case when the dependent variable $X$ is scalar,
i.e. $d=1$. As before, let $(\Omega ,{\mathcal{F}},\mathbb{P})$ be some
nonatomic probability space and $Y$ be some \emph{univariate} random
variable defined on this space. Denoting by $F_{Y}$ the distribution
function of $Y$:
\begin{equation*}
F_{Y}(\alpha ):=\mathbb{P}(Y\leq \alpha ),\;\forall \alpha \in {\mathbb{R}}
\end{equation*}%
the \emph{quantile} function of $Y$, $Q_{Y}=F_{Y}^{-1}$ is the generalized
inverse of $F_{Y}$ given by the formula:
\begin{equation}
Q_{Y}(t):=\inf \{\alpha \in {\mathbb{R}}\;:\;F_{Y}(\alpha )>t\}%
\mbox{ for all
}t\in (0,1).  \label{defqy}
\end{equation}%
Let us now recall two well-known facts about quantiles:

\begin{itemize}
\item $\alpha=Q_Y(t)$ is a solution of the convex minimization problem
\begin{equation}  \label{minquant}
\min_{\alpha} \{{\mathbb{E}}((Y-\alpha)_+)+ \alpha(1-t)\}
\end{equation}

\item there exists a uniformly distributed random variable $U$ such that $%
Y=Q_Y(U)$. Moreover, among uniformly distributed random variables, $U$ is
maximally correlated to $Y$ in the sense that it solves
\begin{equation}  \label{oteasy}
\max \{ {\mathbb{E}}(VY), \; \mathop{\mathrm{Law}}\nolimits(V)=\mu\}
\end{equation}
where $\mu:=\mathop{\mathrm{uniform}}\nolimits([0,1])$ is the uniform
measure on $[0,1]$.

Of course, when $\mathop{\mathrm{Law}}\nolimits(Y)$ has no atom, i.e. when $%
F_Y$ is continuous, $U$ is unique and given by $U=F_Y(Y)$. Problem (\ref%
{oteasy}) is the easiest example of optimal transport problem one can think
of. The decomposition of a random variable $Y$ as the composed of a monotone
nondecreasing function and a uniformly distributed random variable is called
a \emph{polar factorization} of $Y$, the existence of such decompositions
goes back to Ryff \cite{ryff} and the extension to the multivariate case (by
optimal transport) is due to Brenier \cite{brenier}.
\end{itemize}

We therefore see that there are basically two different approaches to study
or estimate quantiles:

\begin{itemize}
\item the \emph{local} or "$t$ by $t$" approach which consists, for a fixed
probability level $t$, in using directly formula (\ref{defqy}) or the
minimization problem (\ref{minquant}) (or some approximation of it), this
can be done very efficiently in practice but has the disadvantage of
forgetting the fundamental global property of the quantile function: it
should be monotone in $t$,

\item the global approach (or polar factorization approach), where quantiles
of $Y$ are defined as all nondecreasing functions $Q$ for which one can
write $Y=Q(U)$ with $U$ uniformly distributed; in this approach, one rather
tries to recover directly the whole monotone function $Q$ (or the uniform
variable $U$ that is maximally correlated to $Y$), in this global approach,
one should rather use the optimization problem (\ref{oteasy}).
\end{itemize}

Let us assume now that, in addition to the random variable $Y$, we are also
given a random vector $X\in {\mathbb{R}}^N$ which we may think of as being a
list of explanatory variables for $Y$. We are therefore interested in the
dependence between $Y$ and $X$ and in particular the conditional quantiles
of $Y$ given $X=x$. In the sequel we shall denote by $\nu$ the joint law of $%
(X,Y)$, $\nu:=\mathop{\mathrm{Law}}\nolimits(X,Y)$ and assume that $\nu$ is
compactly supported on ${\mathbb{R}}^{N+1}$ (i.e. $X$ and $Y$ are bounded).
We shall also denote by $m$ the first marginal of $\nu$ i.e. $m:={\Pi_{X}}%
_\# \nu=\mathop{\mathrm{Law}}\nolimits(X)$. We shall denote by $F(x,y)$ the
conditional cdf:
\begin{equation*}
F(x,y):=\mathbb{P}(Y\le y \vert X=x)
\end{equation*}
and $Q(x,t)$ the conditional quantile
\begin{equation*}
Q(x,t):=\inf\{\alpha\in {\mathbb{R}} \; : \; F(x,\alpha)>t\}.
\end{equation*}
For the sake of simplicity we shall also assume that:

\begin{itemize}
\item for $m$-a.e. $x$, $t\mapsto Q(x,t)$ is continuous and increasing (so
that for $m$-a.e. $x$, identities $Q(x, F(x,y))=y$ and $F(x, Q(x,t))=t$ hold
for every $y$ and every $t$),

\item the law of $(X,Y)$ does not charge nonvertical hyperplanes i.e. for
every $(\alpha, \beta)\in {\mathbb{R}}^{1+N}$, $\mathbb{P}%
(Y=\alpha+\beta\cdot X)=0$.
\end{itemize}

Finally we denote by $\nu^x$ the conditional probability of $Y$ given $X=x$
so that $\nu=m\otimes \nu^x$.

\subsection{A variational characterization of conditional quantiles}

\label{univar1}

Let us define the random variable $U:=F(X,Y)$, then by construction:
\begin{equation*}
\begin{split}
\mathbb{P}(U< t\vert X=x)&=\mathbb{P}(F(x,Y)<t \vert X=x)=\mathbb{P}%
(Y<Q(x,t) \vert X=x) \\
&=F(x,Q(x,t))=t.
\end{split}%
\end{equation*}
From this elementary observation we deduce that

\begin{itemize}
\item $U$ is independent from $X$ (since its conditional cdf does not depend
on $x$),

\item $U$ is uniformly distributed,

\item $Y=Q(X,U)$ where $Q(x,.)$ is increasing.
\end{itemize}

This easy remark leads to a sort of conditional polar factorization of $Y$
with an independence condition between $U$ and $X$. We would like to
emphasize now that there is a variational principle behind this conditional
decomposition. Recall that we have denoted by $\mu$ the uniform measure on $%
[0,1]$. Let us consider the variant of the optimal transport problem (\ref%
{oteasy}) where one further requires $U$ to be independent from the vector
of regressors $X$:
\begin{equation}  \label{otindep1d}
\max \{ {\mathbb{E}}(VY), \; \mathop{\mathrm{Law}}\nolimits(V)=\mu, \; V
\perp \! \! \! \perp X \}.
\end{equation}
which in terms of joint law $\theta=\mathop{\mathrm{Law}}\nolimits(X,Y, U)$
can be written as
\begin{equation}  \label{mk1}
\max_{\theta\in I (\nu, \mu)} \int u\cdot y \; \theta(dx, dy, du)
\end{equation}
where $I(\mu, \nu)$ consists of probability measures $\theta$ on ${\mathbb{R}%
}^{N+1}\times [0,1]$ such that the $(X,Y)$ marginal of $\theta$ is $\nu$ and
the $(X,U)$ marginal of $\theta$ is $m\otimes \mu$. Problem (\ref{mk1}) is a
linear programming problem and our assumptions easily imply that it
possesses solutions, moreover its dual formulation (see \cite{ccg} for
details) reads as the minimization of
\begin{equation}  \label{mk1dual}
\inf J(\varphi, \psi)= \int \varphi(x,u)m(dx) \mu(du)+\int \psi(x,y) \nu(dx,
dy)
\end{equation}
among pairs of potentials $\varphi$, $\psi$ that pointwise satisfy the
constraint
\begin{equation}  \label{constr1}
\varphi(x,u)+\psi(x,y)\ge uy.
\end{equation}
Rewriting $J(\varphi, \psi)$ as
\begin{equation*}
J(\varphi, \psi)= \int \Big( \int \varphi(x,u) \mu(du)+\int \psi(x,y)
\nu^x(dy) \Big) m(dx)
\end{equation*}
and using the fact that the right hand side of the constraint (\ref{constr1}%
) has no dependence in $x$, we observe that (\ref{mk1dual}) can actually be
solved "$x$ by $x$". More precisely, for fixed $x$ in the support of $m$, $%
\varphi(x,.)$ and $\psi(x,.)$ are obtained by solving
\begin{equation*}
\inf \int f(u) \mu(du)+ \int g(y) \nu^x(dy) \; : \; f(u)+g(y)\ge uy
\end{equation*}
which appears naturally in optimal transport and is well-known to admit a
solution which is given by a pair of convex conjugate functions (see \cite%
{villani} \cite{villani2}). In other words, the infimum in (\ref{mk1dual})
is attained by a pair $\varphi$ and $\psi$ such that for $m$-a.e. $x$, $%
\varphi(x,.)$ and $\psi(x,.)$ are conjugate convex functions:
\begin{equation*}
\varphi(x,u)=\sup_{y} \{uy-\psi(x,y)\}, \; \psi(x,y):=\sup_{u}
\{uy-\varphi(x,u)\}.
\end{equation*}
Since $\varphi(x,.)$ is convex it is differentiable a.e. and then $%
\partial_u \varphi(x,u)$ is defined for a.e. $u$, moreover $\partial_u
\varphi(x,.)_\#\mu=\nu^x$; hence $\partial_u \varphi(x,.)$ is a
nondecreasing map which pushes $\mu$ forward to $\nu^x$: it thus coincides
with the conditional quantile
\begin{equation}  \label{quantilopt}
\partial_u \varphi(x,t)=Q(x,t) \mbox{ for $m$-a.e. $x$ and every $t$}.
\end{equation}
We then have the following variational characterization of conditional
quantiles

\begin{thm}
Let $\varphi$ and $\psi$ solve (\ref{mk1dual}). Then for $m$-a.e. $x$, the
conditional quantile $Q(x,.)$ is given by:
\begin{equation*}
Q(x,.) =\partial_u \varphi(x,.)
\end{equation*}
and the conditional cdf $F(x,.)$ is given by:
\begin{equation*}
F(x,.)=\partial_y \psi(x,.).
\end{equation*}

Let now $\theta$ solve (\ref{mk1}), there is a unique $U$ such that $%
\mathop{\mathrm{Law}}\nolimits(X,Y, U)=\theta$ (so that $U$ is uniformly
distributed and independent from $X$) and it is given by $Y=\partial_u
\varphi(X,U)$ almost surely.
\end{thm}

\begin{proof}
The fact that identity \pref{quantilopt} holds for every $t$ and $m$ a.e. $x$ comes from the continuity of the conditional quantile. The second identity  comes from the continuity of the conditional cdf. Now, duality tells us that the maximum in \pref{mk1} coincides with the infimum in \pref{mk1dual}, so that if  $\theta\in I(\mu, \nu)$ is optimal for \pref{mk1dual} and $(\tX, \tY, \tU)$ has law $\theta$\footnote{the fact that there exists such a triple follows from the nonatomicity of the underlying space.},  we have
\[\E(\tU \tY)=\E(\varphi(\tX,\tU)+\psi(\tX,\tY)).\]
Hence, almost surely
\[\tU \tY=\varphi(\tX,\tU)+\psi(\tX,\tY).\]
which, since $\varphi(x,.)$ and $\psi(x,.)$ are conjugate and $\varphi(x,.)$ is differentiable, gives
\begin{equation}\label{sousdiff}
\tY= \partial_u \varphi(\tX, \tU)=Q(\tX, \tU).
\end{equation}
Since $F(x,.)$ is the inverse of the conditional quantile, we can invert the previous relation as
\begin{equation}\label{sousdiff2}
\tU= \partial_y \psi(\tX, \tY)=F(\tX, \tU).
\end{equation}
We then define
\[U:= \partial_y \psi(X, Y)=F(X,Y),\]
then obviously, by construction $\Law(X,Y, U)=\theta$ and $Y=\partial_u \varphi(X,U)=Q(X,U)$ almost surely. If $\Law(X,Y, U)=\theta$ , then as observed above, necessarily $U=F(X,Y)$ which proves the uniqueness claim.
\end{proof}

To sum up, thanks to the two problems (\ref{mk1}) and (\ref{mk1dual}), we
have been able to find a \emph{conditional polar factorization} of $Y$ as
\begin{equation}  \label{polarfact1}
Y=Q(X,U), \; \mbox{ $Q$ nondecreasing in $U$, $U$ uniform, $U \Perp X$}.
\end{equation}
One obtains $U$ thanks to the the correlation maximization with an
independence constraint problem (\ref{otindep1d}) and one obtains the
primitive of $Q(X,.)$ by the dual problem (\ref{mk1dual}).

In this decomposition, it is very demanding to ask that $U$ is independent
from the regressors $X$, in turn, the function $Q(X,.)$ is just monotone
nondecreasing. In practice, the econometrician rather looks for a specific
form of $Q$ (linear in $X$ for instance), which by duality will amount to
relaxing the independence constraint. We shall develop this idea in details
in the next paragraphs and relate it to classical quantile regression.

\subsection{Quantile regression: from specification to quasi-specification}

\label{univar2}

From now on, we normalize $X$ to be centered i.e. assume (and this is
without loss of generality) that
\begin{equation*}
{\mathbb{E}}(X)=0.
\end{equation*}
We also assume that $m:=\mathop{\mathrm{Law}}\nolimits(X)$ is nondegenerate
in the sense that its support contains some ball centered at ${\mathbb{E}}%
(X)=0$.

Since the seminal work of Koenker and Bassett \cite{kb}, it has been widely
accepted that a convenient way to estimate conditional quantiles is to
stipulate an affine form with respect to $x$ for the conditional quantile.
Since a quantile function should be monotone in its second argument, this
leads to the following definition

\begin{defi}
Quantile regression is under correct specification if there exist $(\alpha, \beta)\in C([0,1],
{\mathbb{R}})\times C([0,1], {\mathbb{R}}^N)$ such that for $m$-a.e. $x$
\begin{equation}  \label{monqr}
t\mapsto \alpha(t)+\beta(t)\cdot x \mbox{ is increasing on $[0,1]$}
\end{equation}
and
\begin{equation}  \label{linearcq}
Q(x,t)=\alpha(t)+ x\cdot \beta(t),
\end{equation}
for $m$-a.e. $x$ and every $t\in [0,1]$. If (\ref{monqr})-(\ref{linearcq})
hold, quantile regression is under correct specification with regression coefficients $%
(\alpha, \beta)$.
\end{defi}

Specification of quantile regression can be characterized by

\begin{prop}
Let $(\alpha, \beta)$ be continuous and satisfy (\ref{monqr}). Quantile
regression is under correct specification with regression coefficients $(\alpha, \beta)$ if
and only if there exists $U$ such that
\begin{equation}  \label{polarfqrind}
Y=\alpha(U)+X\cdot \beta(U) \mbox{ a.s.} , \; \mathop{\mathrm{Law}}%
\nolimits(U)=\mu, \; U \perp \! \! \! \perp X.
\end{equation}
\end{prop}

\begin{proof}
The fact that specification of quantile regression implies decomposition \pref{polarfqrind} has already been explained in paragraph \ref{univar1}. Let us assume \pref{polarfqrind}, and compute
\[\begin{split}
F(x, \alpha(t)+\beta(t)\cdot x)&=\PP(Y\leq \alpha(t)+\beta(t) x\vert X=x)\\
&= \PP(\alpha(U)+x\cdot \beta(U) \leq \alpha(t)+\beta(t) x\vert X=x)\\
&=\PP(U\leq t \vert X=x)=\PP(U\le t)=t
\end{split}\]
so that $Q(x,t)=\alpha(t)+\beta(t)\cdot x$.
\end{proof}

Koenker and Bassett showed that, for a fixed probability level $t$, the
regression coefficients $(\alpha, \beta)$ can be estimated by quantile
regression i.e. the minimization problem
\begin{equation}  \label{kb0}
\inf_{(\alpha, \beta) \in {\mathbb{R}}^{1+N}} {\mathbb{E}}(\rho_t(Y-\alpha -
\beta\cdot X))
\end{equation}
where the penalty $\rho_t$ is given by $\rho_t(z) := tz_- +(1-t)z_+$ with $%
z_-$ and $z_+$ denoting the negative and positive parts of $z$. For further
use, note that (\ref{kb0}) can be conveniently be rewritten as
\begin{equation}  \label{kb1}
\inf_{(\alpha, \beta) \in {\mathbb{R}}^{1+N}} \{ {\mathbb{E}}%
((Y-\alpha-\beta \cdot X)_+)+(1-t) \alpha\}.
\end{equation}
As already noticed by Koenker and Bassett, this convex program admits as
dual formulation
\begin{equation}  \label{dt}
\sup \{{\mathbb{E}}(U_t Y)) \; : \; U_t \in [0,1], \; {\mathbb{E}}%
(U_t)=(1-t), \; {\mathbb{E}}(U_t X)=0 \}.
\end{equation}
An optimal $(\alpha, \beta)$ for (\ref{kb1}) and an optimal $U_t$ in (\ref%
{dt}) are related by the complementary slackness condition:
\begin{equation}
Y>\alpha +\beta \cdot X \Rightarrow U_t=1, \mbox{ and } \; Y<\alpha + \beta
\cdot X \Rightarrow U_t=0.
\end{equation}
Note that $\alpha$ appears naturally as a Lagrange multiplier associated to
the constraint ${\mathbb{E}}(U_t)=(1-t)$ and $\beta$ as a Lagrange
multiplier associated to ${\mathbb{E}}(U_t X)=0$. Since $\nu=%
\mathop{\mathrm{Law}}\nolimits(X,Y)$ gives zero mass to nonvertical
hyperplanes, we may simply write
\begin{equation}  \label{frombetatoU}
U_t=\mathbf{1}_{\{Y>\alpha +\beta \cdot X\}}
\end{equation}
and thus the constraints ${\mathbb{E}}(U_t)=(1-t)$, ${\mathbb{E}}(XU_t)=0$
read
\begin{equation}  \label{normaleq}
{\mathbb{E}}( \mathbf{1}_{\{Y> \alpha+ \beta \cdot X \}})=\mathbb{P}(Y>
\alpha+ \beta \cdot X) = (1-t),\; {\mathbb{E}}(X \mathbf{1}_{\{Y> \alpha+
\beta \cdot X \}} ) =0
\end{equation}
which simply are the first-order conditions for (\ref{kb1}).

Any pair $(\alpha, \beta)$ which solves\footnote{%
Uniqueness will be discussed later on.} the optimality conditions (\ref%
{normaleq}) for the Koenker and Bassett approach will be denoted
\begin{equation*}
\alpha=\alpha^{QR}(t), \beta=\beta^{QR}(t)
\end{equation*}
and the variable $U_t$ solving (\ref{dt}) given by (\ref{frombetatoU}) will
similarly be denoted $U_t^{QR}$
\begin{equation}  \label{utqr}
U_t^{QR}:=\mathbf{1}_{\{Y>\alpha^{QR}(t) +\beta^{QR}(t) \cdot X\}}.
\end{equation}

Note that in the previous considerations the probability level $t$ is fixed,
this is what we called the "$t$ by $t$" approach. For this approach to be
consistent with conditional quantile estimation, if we allow $t$ to vary we
should add an additional monotonicity requirement:

\begin{defi}
Quantile regression is under quasi-specification if there exists for each $t$, a
solution $(\alpha^{QR}(t), \beta^{QR}(t))$ of (\ref{normaleq}) (equivalently
the minimization problem (\ref{kb0})) such that $t\in [0,1]\mapsto
(\alpha^{QR}(t), \beta^{QR}(t))$ is continuous and, for $m$-a.e. $x$
\begin{equation}  \label{monqrqs}
t\mapsto \alpha^{QR}(t)+\beta^{QR}(t)\cdot x \mbox{ is increasing on $[0,1]$}%
.
\end{equation}
\end{defi}

A first consequence of quasi-specification is given by

\begin{prop}
\label{qrqsdec} If quantile regression is under quasi-specification and if we define $%
U^{QR}:=\int_0^1 U_t^{QR} dt$ (recall that $U_t^{QR}$ is given by (\ref{utqr}%
)) then:

\begin{itemize}
\item $U^{QR}$ is uniformly distributed,

\item $X$ is mean-independent from $U^{QR}$ i.e. ${\mathbb{E}}(X\vert
U^{QR})={\mathbb{E}}(X)=0$,

\item $Y=\alpha^{QR}(U^{QR})+ \beta^{QR}(U^{QR})\cdot X$ a.s.
\end{itemize}

Moreover $U^{QR}$ solves the correlation maximization problem with a
mean-independence constraint:

\begin{equation}  \label{maxcorrmi}
\max \{ {\mathbb{E}}(VY), \; \mathop{\mathrm{Law}}\nolimits(V)=\mu, \; {%
\mathbb{E}}(X\vert V)=0\}.
\end{equation}
\end{prop}

\begin{proof}
Obviously
\[U_t^{QR}=1\Rightarrow U^{QR} \ge t, \mbox{ and }  \; U^{QR}>t \Rightarrow U_t^{QR}=1\]
hence $\PP(U^{QR}\ge t)\ge \PP(U_t^{QR}=1)=\PP(Y> \alpha^{QR}(t)+\beta^{QR}(t)\cdot X)=(1-t)$ and  $\PP(U^{QR}> t)\le \PP(U_t^{QR}=1)=(1-t)$ which proves  that $U^{QR}$ is uniformly distributed and $\{U^{QR}>t\}$ coincides with $\{U^{QR}_t=1\}$ up to a set of null probability. We thus have $\E(X \indi_{U^{QR}>t})=\E(X U_t^{QR})=0$, by a standard approximation argument we deduce that $\E(Xf(U^{QR}))=0$ for every $f\in C([0,1], \R)$ which means that $X$ is mean-independent from $U^{QR}$.

As already observed $U^{QR}>t$ implies that $Y>\alpha^{QR}(t)+\beta^{QR}(t)\cdot X$ in particular $Y\ge \alpha^{QR}(U^{QR}-\delta)+\beta^{QR}(U^{QR}- \delta) \cdot X$ for $\delta>0$, letting $\delta\to 0^+$ and using the continuity of $(\alpha^{QR}, \beta^{QR})$ we get $Y\ge \alpha^{QR}(U^{QR})+\beta^{QR}(U^{QR}) \cdot X$. The converse inequality is obtained similarly by remaking that $U^{QR}<t$ implies that $Y\le \alpha^{QR}(t)+\beta^{QR}(t)\cdot X$.

Let us now prove that $U^{QR}$ solves \pref{maxcorrmi}.  Take $V$ uniformly distributed, such that $X$ is mean-independent from $V$ and set $V_t:=\indi_{\{V>t \}}$, we then have $\E(X V_t)=0$, $\E(V_t)=(1-t)$ but since $U_t^{QR}$ solves \pref{dt} we have $\E(V_t Y)\le \E(U_t^{QR}Y)$. Observing that $V=\int_0^1 V_t dt$ and integrating the previous inequality with respect to $t$ gives $\E(VY)\le \E(U^{QR}Y)$ so that $U^{QR}$ solves \pref{maxcorrmi}.

\end{proof}

Let us continue with a uniqueness argument for the mean-independent
decomposition given in proposition \ref{qrqsdec}:

\begin{prop}
\label{uniquedec} Let us assume that
\begin{equation*}
Y=\alpha(U)+\beta(U)\cdot X=\overline{\alpha} (\overline{U})+ \overline{\beta%
}(\overline{U})\cdot X
\end{equation*}
with:

\begin{itemize}
\item both $U$ and $\overline{U}$ uniformly distributed,

\item $X$ is mean-independent from $U$ and $\overline{U}$: ${\mathbb{E}}%
(X\vert U)={\mathbb{E}}(X\vert \overline{U})=0$,

\item $\alpha, \beta, \overline{\alpha}, \overline{\beta}$ are continuous on
$[0,1]$,

\item $(\alpha, \beta)$ and $(\overline{\alpha}, \overline{\beta})$ satisfy
the monotonicity condition (\ref{monqr}),
\end{itemize}

then
\begin{equation*}
\alpha=\overline{\alpha}, \; \beta=\overline{\beta}, \; U=\overline{U}.
\end{equation*}
\end{prop}

\begin{proof}
Let us define for every $t\in [0,1]$
\[\varphi(t):=\int_0^t \alpha(s)ds, \; b(t):=\int_0^t \beta(s)ds.\]
Let us also define for $(x,y)$ in $\R^{N+1}$:
\[\psi(x,y):=\max_{t\in [0,1]} \{ty-\varphi(t)-b(t)\cdot x\}\]
thanks to monotonicity condition \pref{monqr}, the maximization program above is strictly concave in $t$ for every $y$ and $m$-a.e.$x$. We then remark that $Y=\alpha(U)+\beta(U)\cdot X=\varphi'(U)+b'(U)\cdot X$ exactly is the first-order condition for the above maximization problem when $(x,y)=(X,Y)$. In other words, we have
\begin{equation}\label{ineqq}
\psi(x,y)+b(t)\cdot x + \varphi(t)\ge ty, \; \forall (t,x,y)\in [0,1]\times \R^N\times \R
\end{equation}
with and equality for $(x,y,t)=(X,Y,U)$ i.e.
\begin{equation}\label{eqas}
\psi(X,Y)+b(U)\cdot X + \varphi(U)=UY, \;   \mbox{ a.s. }
\end{equation}
Using the fact that $\Law(U)=\Law(\oU)$ and the fact that mean-independence gives $\E(b(U)\cdot X)=\E(b(\oU)\cdot X)=0$, we have
\[\E(UY)=\E( \psi(X,Y)+b(U)\cdot X + \varphi(U))= \E( \psi(X,Y)+b(\oU)\cdot X + \varphi(\oU)) \ge \E(\oU Y)\]
but reversing the role of $U$ and $\oU$, we also have $\E(UY)\le \E(\oU Y)$ and then
\[ \E(\oU Y)=  \E( \psi(X,Y)+b(\oU)\cdot X + \varphi(\oU))\]
so that, thanks to inequality \pref{ineqq}
\[\psi(X,Y)+b(\oU)\cdot X + \varphi(\oU)=\oU Y, \;   \mbox{ a.s. }\]
which means that $\oU$ solves $\max_{t\in [0,1]} \{tY-\varphi(t)-b(t)\cdot X\}$ which, by strict concavity admits $U$ as unique solution. This proves that $U=\oU$ and thus
\[\alpha(U)-\oal(U)=(\obe(U)-\beta(U))\cdot X\]
taking the conditional expectation  of both sides with respect to $U$, we then obtain $\alpha=\oal$ and thus $\beta(U)\cdot X=\obe(U)\cdot X$ a.s.. We then compute
\[\begin{split}
F(x, \alpha(t)+\beta(t)\cdot x)&= \PP(\alpha(U)+\beta(U)\cdot X \le \alpha(t)+\beta(t)\cdot x     \vert X=x)  \\
&=\PP( \alpha(U)+ \beta(U)\cdot x \le \alpha(t)+\beta(t)\cdot x \vert X=x)\\
&=\PP(U\le t \vert X=x)
\end{split}\]
and similarly $F(x, \alpha(t)+\obe(t)\cdot x)=\PP(U\le t \vert X=x)=F(x, \alpha(t)+\beta(t)\cdot x)$. Since $F(x,.)$ is increasing for $m$-a.e. $x$, we deduce that $\beta(t)\cdot x=\obe(t)\cdot x$ for $m$-a.e. $x$ and every $t\in[0,1]$. Finally,  the previous considerations and the nondegeneracy of $m$ enable us to conclude that $\beta=\obe$.

\end{proof}

\begin{coro}
If quantile regression is under quasi-specification, the regression coefficients $%
(\alpha^{QR}, \beta^{QR})$ are uniquely defined and if $Y$ can be written as
\begin{equation*}
Y=\alpha(U)+\beta(U)\cdot X
\end{equation*}
for $U$ uniformly distributed, $X$ being mean independent from $U$, $%
(\alpha, \beta)$ continuous such that the monotonicity condition (\ref{monqr}%
) holds then necessarily
\begin{equation*}
\alpha=\alpha^{QR}, \; \beta=\beta^{QR}.
\end{equation*}
\end{coro}

To sum up, we have shown that quasi-specification is equivalent to the
validity of the factor linear model:
\begin{equation*}
Y=\alpha(U)+\beta(U)\cdot X
\end{equation*}
for $(\alpha, \beta)$ continuous and satisfying the monotonicity condition (%
\ref{monqr}) and $U$, uniformly distributed and such that $X$ is
mean-independent from $U$. This has to be compared with the decomposition of
paragraph \ref{univar1} where $U$ is required to be independent from $X$ but
the dependence of $Y$ with respect to $U$, given $X$, is given by any
nondecreasing function of $U$.

\subsection{Global approaches and duality}

\label{univar3}

Now we wish to address quantile regression in the case where neither
specification nor quasi-specification can be taken for granted. In such a
general situation, keeping in mind the remarks from the previous paragraphs,
we can think of two natural approaches.

The first one consists in studying directly the correlation maximization
with a mean-independence constraint (\ref{maxcorrmi}). The second one
consists in getting back to the Koenker and Bassett $t$ by $t$ problem (\ref%
{dt}) but adding as an additional global consistency constraint that $U_t$
should be nonincreasing with respect to $t$:

\begin{equation}  \label{monconstr}
\sup\{{\mathbb{E}}(\int_0^1 U_t Ydt ) \; : \: U_t \mbox{ nonincr.}, U_t\in
[0,1],\; {\mathbb{E}}(U_t)=(1-t), \; {\mathbb{E}}(U_t X)=0\}
\end{equation}

Our aim is to compare these two approaches (and in particular to show that
the maximization problems (\ref{maxcorrmi}) and (\ref{monconstr}) have the
same value) as well as their dual formulations. Before going further, let us
remark that (\ref{maxcorrmi}) can directly be considered in the multivariate
case whereas the monotonicity constrained problem (\ref{monconstr}) makes
sense only in the univariate case.

As proven in \cite{ccg}, (\ref{maxcorrmi}) is dual to
\begin{equation}  \label{dualmi}
\inf_{(\psi, \varphi, b)} \{{\mathbb{E}}(\psi(X,Y))+{\mathbb{E}}(\varphi(U))
\; : \; \psi(x,y)+ \varphi(u)\ge uy -b(u)\cdot x\}
\end{equation}
which can be reformulated as:
\begin{equation}  \label{dualmiref}
\inf_{(\varphi, b)} \int \max_{t\in [0,1]} ( ty- \varphi(t) -b(t)\cdot x)
\nu(dx, dy) +\int_0^1 \varphi(t) dt
\end{equation}
in the sense that
\begin{equation}  \label{nodualgap}
\sup (\ref{maxcorrmi})=\inf(\ref{dualmi})=\inf (\ref{dualmiref}).
\end{equation}

The existence of a solution to (\ref{dualmi}) is not straightforward and is
established under appropriate assumptions in the appendix directly in the
multivariate case. The following result shows that there is a $t$-dependent
reformulation of (\ref{maxcorrmi}):

\begin{lem}
\label{treform} The value of (\ref{maxcorrmi}) coincides with
\begin{equation}  \label{monconstr01}
\sup\{{\mathbb{E}}(\int_0^t U_t Ydt ) \; : \: U_t \mbox{ nonincr.}, U_t\in
\{0,1\},\; {\mathbb{E}}(U_t)=(1-t), \; {\mathbb{E}}(U_t X)=0\}
\end{equation}
\end{lem}

\begin{proof}
Let $U$ be admissible for \pref{maxcorrmi} and define $U_t:=\indi_{\{U>t\}}$ then $U=\int_0^1 U_t dt$ and obviously $(U_t)_t$ is admissible for \pref{monconstr01}, we thus have $\sup \pref{maxcorrmi} \le \sup \pref{monconstr01}$. Take now $(V_t)_t$ admissible for  \pref{monconstr01} and let $V:=\int_0^1 V_t dt$, we then have
\[V>t \Rightarrow V_t=1\Rightarrow V\ge t\]
since $\E(V_t)=(1-t)$ this implies that $V$ is uniformly distributed and $V_t=\indi_{\{V>t\}}$ a.s. so that $\E(X   \indi_{\{V>t\}})=0$ which implies that $X$ is mean-independent from $V$ and thus $\E(\int_0^1 V_t Y dt)\le \sup \pref{maxcorrmi}$. We conclude that   $\sup \pref{maxcorrmi} = \sup \pref{monconstr01}$.

\end{proof}

Let us now define
\begin{equation*}
\mathcal{C}:=\{u \; : \; [0,1]\mapsto [0,1], \mbox { nonincreasing}\}
\end{equation*}

Let $(U_t)_t$ be admissible for (\ref{monconstr}) and set
\begin{equation*}
v_t(x,y):={\mathbb{E}}(U_t \vert X=x, Y=y), \; V_t:= v_t(X,Y)
\end{equation*}
it is obvious that $(V_t)_t$ is admissible for (\ref{monconstr}) and by
construction ${\mathbb{E}}(V_t Y)={\mathbb{E}}(U_t Y)$. Moreover the
deterministic function $(t,x,y)\mapsto v_t(x,y)$ satisfies the following
conditions:
\begin{equation}  \label{CCt}
\mbox{for fixed $(x,y)$, } t\mapsto v_t(x,y) \mbox{ belongs to $\CC$,}
\end{equation}
and for a.e. $t\in [0,1]$,
\begin{equation}  \label{moments}
\int v_t(x,y) \nu(dx, dy)=(1-t), \; \int v_t(x,y) x\nu(dx, dy)=0.
\end{equation}
Conversely, if $(t,x,y)\mapsto v_t(x,y)$ satisfies (\ref{CCt})-(\ref{moments}%
), $V_t:=v_t(X,Y)$ is admissible for (\ref{monconstr}) and ${\mathbb{E}}(V_t
Y)=\int v_t(x,y) y \nu(dx, dy)$. All this proves that $\sup(\ref{monconstr})$
coincides with
\begin{equation}  \label{supvt}
\sup_{(t,x,y)\mapsto v_t(x,y)} \int v_t(x,y) y \nu(dx, dy)dt
\mbox{ subject
to: } (\ref{CCt})-(\ref{moments})
\end{equation}

\begin{thm}
\label{equivkb}
\begin{equation*}
\sup (\ref{maxcorrmi})=\sup (\ref{monconstr}).
\end{equation*}
\end{thm}

\begin{proof}

We know from lemma \ref{treform} and the remarks above that
\[\sup \pref{maxcorrmi}=\sup \pref{monconstr01} \le  \sup \pref{monconstr}=\sup \pref{supvt}.\]
But now we may  get rid of constraints \pref{moments} by rewriting \pref{supvt} in sup-inf form as
\[\begin{split}
\sup_{\qtext{$v_t$ satisfies \pref{CCt}}}  \inf_{(\alpha, \beta)} \int v_t(x,y)(y-\alpha(t)-\beta(t) \cdot x) \nu(dx,dy)dt +\int_0^1 (1-t)\alpha(t) dt.
\end{split}\]
Recall that one always have $\sup \inf \le  \inf \sup$ so that $\sup\pref{supvt}$ is less than
\[\begin{split}
\inf_{(\alpha, \beta)} \sup_{\qtext{$v_t$ satisf. \pref{CCt}}}  \int v_t(x,y)(y-\alpha(t)-\beta(t) \cdot x) \nu(dx,dy)dt +\int_0^1 (1-t)\alpha(t) dt\\
\le \inf_{(\alpha, \beta)}  \int \Big  (\sup_{v\in \CC}  \int_0^1 v(t)(y-\alpha(t)-\beta(t)x)dt    \Big) \nu(dx,dy)+ \int_0^1 (1-t)\alpha(t) dt.
\end{split}\]
It follows from Lemma \ref{suppC} below that, for $q\in L^1(0,1)$ defining $Q(t):=\int_0^t q(s) ds$, one has
\[\sup_{v\in \CC}  \int_0^1 v(t) q(t)dt=\max_{t\in [0,1]} Q(t).\]
So setting $\varphi(t):=\int_0^t \alpha(s) ds$, $b(t):=\int_0^t \beta(s)ds$ and remarking that integrating by parts immediately gives
\[\int_0^1 (1-t)\alpha(t) dt=\int_0^1 \varphi(t) dt\]
we thus have
\[\begin{split}
\sup_{v\in \CC}  \int_0^1 v(t)(y-\alpha(t)-\beta(t)x)dt   + \int_0^1 (1-t)\alpha(t) dt\\
=   \max_{t\in[0,1]} \{t y-\varphi(t)-b(t) x\} +\int_0^1 \varphi(t) dt.
\end{split}\]
This yields\footnote{The functions $\varphi$ and $b$ constructed above vanish at $0$ and  are absolutely continuous  but this is by no means a restriction in the minimization problem \pref{dualmiref} as explained in paragraph \ref{mvqr2}.}
\[\sup\pref{supvt} \le \inf_{(\varphi, b)}    \int \max_{t\in [0,1]} ( ty-  \varphi(t) -b(t)\cdot x)   \nu(dx, dy) +\int_0^1 \varphi(t) dt =\inf \pref{dualmiref}    \]
but we know from \pref{nodualgap} that $\inf \pref{dualmiref} =\sup \pref{maxcorrmi}$ which ends the proof.

\end{proof}

In the previous proof, we have used the elementary result (proven in the
appendix).

\begin{lem}
\label{suppC} Let $q\in L^1(0,1)$ and define $Q(t):=\int_0^t q(s) ds$ for
every $t\in [0,1]$, one has
\begin{equation*}
\sup_{v\in \mathcal{C}} \int_0^1 v(t) q(t)dt=\max_{t\in [0,1]} Q(t).
\end{equation*}
\end{lem}

\section*{Appendix}

\subsection*{Proof of Lemma \protect\ref{suppC}}

Since $\mathbf{1}_{[0,t]} \in \mathcal{C}$, one obviously first has
\begin{equation*}
\sup_{v\in \mathcal{C}} \int_0^1 v(s) q(s)ds \ge \max_{t\in [0,1]} \int_0^t
q(s)ds=\max_{t\in [0,1]} Q(t).
\end{equation*}
Let us now prove the converse inequality, taking an arbitrary $v\in \mathcal{%
C}$. We first observe that $Q$ is absolutely continuous and that $v$ is of
bounded variation (its derivative in the sense of distributions being a
bounded nonpositive measure which we denote by $\eta$), integrating by parts
and using the definition of $\mathcal{C}$ then give:
\begin{equation*}
\begin{split}
\int_0^1 v(s) q(s)ds &=-\int_0^1 Q \eta + v(1^-) Q(1) \\
& \le (\max_{[0,1]} Q)\times (-\eta([0,1]) + v(1^-) Q(1) \\
&= (\max_{[0,1]} Q) (v(0^+)-v(1^-)) +v(1^-) Q(1) \\
&= (\max_{[0,1]} Q) v(0^+) + (Q(1)- \max_{[0,1]} Q) v(1^-) \\
& \le \max_{[0,1]} Q.
\end{split}%
\end{equation*}

\subsection*{Proof of theorem \protect\ref{existdual}}

Let us denote by $(0, \overline{y})$ the barycenter of $\nu$:
\begin{equation*}
\int_{\Omega} x \; \nu(dx, dy)=0, \; \int_{\Omega} y \; \nu(dx, dy)=:%
\overline{y}
\end{equation*}
and observe that $(0, \overline{y})\in \Omega$ (otherwise, by convexity, $%
\nu $ would be supported on $\partial \Omega$ which would contradict our
assumption that $\nu\in L^{\infty}(\Omega)$).

We wish to prove the existence of optimal potentials for the problem
\begin{equation}  \label{duallike}
\inf_{\psi, \varphi, b } \int_{\Omega} \psi(x,y) d \nu(x,y) + \int_{[0,1]^d}
\varphi(u) d\mu(u)
\end{equation}
subject to the pointwise constraint that
\begin{equation}  \label{const}
\psi(x,y)+\varphi(u)\ge u\cdot y -b(u)\cdot x, \; (x,y)\in \overline{\Omega}%
, \; u\in [0,1]^d.
\end{equation}
Of course, we can take $\psi$ that satisfies
\begin{equation*}
\psi(x,y):=\sup_{u\in [0,1]^d} \{ u\cdot y -b(u)\cdot x-\varphi(u)\}
\end{equation*}
so that $\psi$ can be chosen convex and $1$ Lipschitz with respect to $y$.
In particular, we have
\begin{equation}  \label{lip}
\psi(x,\overline{y})-\vert y-\overline{y} \vert \le \psi(x,y) \leq \psi(x,%
\overline{y})+\vert y-\overline{y} \vert.
\end{equation}
The problem being invariant by the transform $(\psi, \varphi)\to (\psi+C,
\psi-C)$ ($C$ being an arbitrary constant), we can add as a normalization
the condition that
\begin{equation}  \label{normaliz}
\psi(0, \overline{y})=0.
\end{equation}
This normalization and the constraint (\ref{const}) imply that
\begin{equation}  \label{fipos}
\varphi(t)\ge t\cdot \overline{y} -\psi(0, \overline{y}) \ge -\vert
\overline{y} \vert.
\end{equation}
We note that there is one extra invariance of the problem: if one adds an
affine term $q\cdot x$ to $\psi$ this does not change the cost and neither
does it affect the constraint, provided one modifies $b$ accordingly by
substracting to it the constant vector $q$. Take then $q$ in the
subdifferential of $x\mapsto \psi(x,\overline{y})$ at $0$ and change $\psi$
into $\psi-q\cdot x$, we obtain a new potential with the same properties as
above and with the additional property that $\psi(.,\overline{y})$ is
minimal at $x=0$, and thus $\psi(x,\overline{y})\ge 0$, together with (\ref%
{lip}) this gives the lower bound
\begin{equation}  \label{lb}
\psi(x,y)\ge -\vert y-\overline{y} \vert \ge -C
\end{equation}
where the bound comes from the boundedness of $\Omega$ (from now one, $C$
will denote a generic constant maybe changing from one line to another).

Now take a minimizing sequence $(\psi_n, \varphi_n, b_n)\in C(\overline{%
\Omega}, {\mathbb{R}})\times C([0,1]^d, {\mathbb{R}})\times C([0,1]^d, {%
\mathbb{R}}^N)$ where for each $n$, $\psi_n$ has been chosen with the same
properties as above. Since $\varphi_n$ and $\psi_n$ are bounded from below ($%
\varphi_n \ge -\vert \overline{y}\vert$ and $\psi_n \ge C$) and since the
sequence is minimizing, we deduce immediately that $\psi_n$ and $\varphi_n$
are bounded sequences in $L^1$. Let $z=(x,y)\in \Omega$ and $r>0$ be such
that the distance between $z$ and the complement of $\Omega$ is at least $2r$%
, (so that $B_r(z)$ is in the set of points that are at least at distance $r$
from $\partial \Omega$), by assumption there is an $\alpha_r>0$ such that $%
\nu \ge \alpha_r$ on $B_r(z)$. We then deduce from the convexity of $\psi_n$%
:
\begin{equation*}
C \le \psi_n(z)\le \frac{1}{\vert B_r(z)\vert }\int_{B_r(z)} \psi_n \leq
\frac{1}{\vert B_r(z)\vert \alpha_r } \int_{B_r(z)} \vert \psi_n\vert \nu
\le \frac{1}{\vert B_r(z)\vert \alpha_r } \Vert \psi_n \Vert_{L^1(\nu)}
\end{equation*}
so that $\psi_n$ is actually bounded in $L^{\infty}_{{\mathrm{loc}}}$ and by
convexity, we also have
\begin{equation*}
\Vert \nabla \psi_n \Vert_{L^{\infty}(B_r(z))} \le \frac{2}{R-r} \Vert
\psi_n \Vert_{L^{\infty}(B_R(z))}
\end{equation*}
whenever $R>r$ and $B_R(z)\subset \Omega$ (see for instance Lemma 5.1 in
\cite{cg} for a proof of such bounds). We can thus conclude that $\psi_n$ is
also locally uniformly Lipschitz. Therefore, thanks to Ascoli's theorem, we
can assume, taking a subsequence if necessary, that $\psi_n$ converges
locally uniformly to some potential $\psi$.

Let us now prove that $b_n$ is bounded in $L^1$, for this take $r>0$ such
that $B_{2r}(0, \overline{y})$ is included in $\Omega$. For every $x\in
B_r(0)$, any $t\in[0,1]^d $ and any $n$ we then have
\begin{equation*}
\begin{split}
-b_n(t) \cdot x \le \varphi_n(t)-t \cdot \overline{y} + \Vert \psi_n
\Vert_{L^{\infty} (B_r(0, \overline{y}))} \le C+\varphi_n(t)
\end{split}%
\end{equation*}
maximizing in $x \in B_r(0)$ immediately gives
\begin{equation*}
\vert b_n(t )\vert r \leq C +\varphi_n(t).
\end{equation*}
From which we deduce that $b_n$ is bounded in $L^1$ since $\varphi_n$ is.

\smallskip

From Komlos' theorem (see \cite{komlos}), we may find a subsequence such
that the Cesaro means
\begin{equation*}
\frac{1}{n} \sum_{k=1}^n \varphi_k, \; \frac{1}{n} \sum_{k=1}^n b_k
\end{equation*}
converge a.e. respectively to some $\varphi$ and $b$. Clearly $\psi$, $%
\varphi$ and $b$ satisfy the linear constraint (\ref{const}), and since the
sequence of Cesaro means $(\psi^{\prime }_n, \phi^{\prime }_n, b^{\prime
}_n):= n^{-1}\sum_{k=1}^n (\psi_k, \phi_k, b_k)$ is also minimizing, we
deduce from Fatous' Lemma
\begin{equation*}
\begin{split}
&\int_{\Omega} \psi(x,y) d \nu(x,y) + \int_{[0,1]^d} \varphi(u) d\mu(u) \\
& \leq \liminf_n \int_{\Omega} \psi^{\prime }_n(x,y) d \nu(x,y) +
\int_{[0,1]^d} \varphi^{\prime }_n(u) d\mu(u)=\inf(\ref{duallike})
\end{split}%
\end{equation*}
which ends the existence proof.


\begin{thebibliography}{99}
\bibitem{belloni} A. Belloni, R.L. Winkler, On Multivariate Quantiles Under
Partial Orders, The Annals of Statistics, \textbf{39} (2), 1125-1179 (2011).

\bibitem{brenier} Y. Brenier, Polar factorization and monotone rearrangement
of vector-valued functions, Comm. Pure Appl. Math. 44 \textbf{\ 4}, 375--417
(1991).

\bibitem{cg} G. Carlier, A. Galichon, Exponential convergence for a
convexifying equation, ESAIM, Control, Optimisation and Calculus of
Variations, \textbf{18} (3), 611-620 (2012).

\bibitem{ccg} G. Carlier, V. Chernozhukov, A. Galichon, Vector quantile
regression: an optimal transport approach, The Annals of Statistics, \textbf{%
44} (3), 1165-1192 (2016)

\bibitem{egh} I. Ekeland, A. Galichon, M. Henry, Comonotonic measures of
multivariate risks, Math. Finance, \textbf{22} (1), 109-132 (2012).

\bibitem{ektem} I. Ekeland, R. Temam, \textit{Convex Analysis and
Variational Problems}, Classics in Mathematics, Society for Industrial and
Applied Mathematics, Philadelphia, (1999).

\bibitem{otme} A. Galichon, \textit{Optimal Transport Methods in Economics},
Princeton University Press.

\bibitem{gh} A. Galichon, M. Henry, Dual theory of choice with multivariate
risks, J. Econ Theory, \textbf{47} (4), 1501-516 (2012).

\bibitem{hallin} M. Hallin, D. Paindaveine, M. Siman, Multivariate quantiles
and multiple-output regression quantiles: From $L^1$ optimization to
halfspace depth, The Annals of Statistics, \textbf{38} (2), 635-669 (2010).

\bibitem{kb} R. Koenker, G. Bassett, Regression Quantiles, Econometrica,
\textbf{46}, 33-50 (1978).

\bibitem{komlos} J. Komlos, A generalization of a problem of Steinhaus,
\emph{Acta Mathematica Academiae Scientiarum Hungaricae}, \textbf{18}
(1--2), 1967, pp. 217--229.

\bibitem{PuccettiScarsini} G. Puccetti, M. Scarsini, Multivariate
comonotonicity, \emph{Journal of Multivariate Analysis} 101, 291--304 (2010).

\bibitem{ryff} J.V.Ryff , Measure preserving Transformations and
Rearrangements, \textit{J. Math. Anal. and Applications} \textbf{31} (1970),
449--458.

\bibitem{sanbook} F. Santambrogio, \textit{Optimal Transport for Applied
Mathematicians}, Progress in Nonlinear Differential Equations and Their
Applications 87, Birkh\"auser Basel, 2015.

\bibitem{villani} C. Villani, \textit{Topics in optimal transportation},
Graduate Studies in Mathematics, 58, American Mathematical Society,
Providence, RI, 2003.

\bibitem{villani2} C. Villani, \textit{Optimal transport: Old and New},
Grundlehren der mathematischen Wissenschaften, Springer-Verlag, Heidelberg,
2009.
\end{thebibliography}
\end{document}